\documentclass{kms-b}

\issueinfo{}
  {}
  {}
  {}
\pagespan{1}{}

\usepackage{graphicx}
\allowdisplaybreaks

\theoremstyle{plain}
\newtheorem{theorem}{Theorem}[section]

\newtheorem{corollary}[theorem]{Corollary}

\theoremstyle{definition}

\theoremstyle{remark}

\begin{document}

\title[Analytic Properties of the Sum $B_{1}(h,k)$]
{Analytic Properties of the Sum $B_{1}(h,k)$}

\author[E. Cetin]{Elif Cetin}
\address{Elif Cetin \\ Department of Mathematics \\ Uludag University \\ Bursa 16059, Turkey}
\email{elifc2@gmail.com}


\begin{abstract}
In \cite{csc}, Cetin et al. defined a new special finite sum which is
denoted by $B_{1}(h,k)$. In this paper, with the help of the Hardy and
Dedekind sums we will give many properties of the sum $B_{1}(h,k).$ Then we
will give the connections of this sum with the other well-known finite sums
such as the Dedekind sums, the Hardy sums, the Simsek sums $Y(h,k)$ and the
sum $C_{1}(h,k)$. By using the Fibonacci numbers and
two-term polynomial relation, we will also give a new property of the sum $B_{1}(h,k)$.
\end{abstract}

\maketitle

\section{Introduction}

The motivation of this paper is to investigate and study of the properties of the sums $B_{1}(h,k)$ which is defined by
\begin{equation}
	B_{1}(h,k)=\sum_{j=1}^{k-1}\left[ \frac{hj}{k}\right] \left( -1\right) ^{\left[ \frac{hj}{k}\right]}
\end{equation} 
where $ \left[ x\right] $ is the greatest
integer function which can be defined by the help of the sawtooth function $ \left( \left( x\right) \right) $, as follows:

\begin{equation*}
	\left( \left( x\right) \right) =\left\{ 
	\begin{array}{cc}
		x-\left[ x\right] -1/2 & \text{if }x\text{ is not an integer} \\ 
		0 & \text{if }x\text{ is an integer.}%
	\end{array}%
	\right.
\end{equation*}
Dedekind sums, which defined by Richard Dedekind in the nineteeth century,
is given with the below equality:%
\begin{equation*}
	s(h,k)=\sum_{j=1}^{k-1}\left( \left( \frac{hj}{k}\right) \right) \left(
	\left( \frac{j}{k}\right) \right) ,
\end{equation*}%
where $h$ is an integer, $k$ is a positive integer. The basic introduction
to the arithmetic properties of the Dedekind sum is \cite{Hans R.}. Dedekind
defined these sums with the help of the famous Dedekind eta function.
Although Dedekind sums arise in the transformation formula for the eta
function, they can be defined independently of the eta function. Dedekind
sums have many interesting properties but most important one is the
reciprocity theorem: When $h$ and $k$ are coprime positive integers, the
following reciprocity law holds \cite{dedekind}:%
\begin{equation}
	s(h,k)+s(k,h)=-\frac{1}{4}+\frac{1}{12}\left( \frac{h}{k}+\frac{k}{h}+\frac{1%
	}{hk}\right) .  \label{eq1}
\end{equation}%
The first proof of (\ref{eq1}) was given by Richard Dedekind in $1892$ \cite%
{dedekind}. After R. Dedekind, Apostol \cite{Apostol} and many authors have
given many different proofs \cite{Hans R.}. By using contour integration, in 
$1905$, Hardy, \cite{Hardy}, gave another proof of the reciprocity theorem.
In that work, Hardy also gave some finite arithmetical sums which are called
Hardy sums. These Hardy sums are also related to the Dedekind sums and have
many useful properties.

We are ready to recall the Hardy sums which are needed in the further
sections: If $h$ and $k\in \mathbb{Z}$ with $k>0$, the Hardy sums are
defined by%
\begin{eqnarray}
	S(h,k) &=&\sum\limits_{j\text{mod}k}(-1)^{j+1+[\frac{jh}{k}]},  \notag \\
	s_{1}(h,k) &=&\sum\limits_{j\text{mod}k}(-1)^{[\frac{jh}{k}]}\left( \left( 
	\frac{j}{k}\right) \right) ,  \notag \\
	s_{2}(h,k) &=&\sum\limits_{j\text{mod}k}(-1)^{j}\left( \left( \frac{j}{k}%
	\right) \right) \left( \left( \frac{jh}{k}\right) \right) ,  \label{eq2} \\
	s_{3}(h,k) &=&\sum\limits_{j\text{mod}k}(-1)^{j}\left( \left( \frac{jh}{k}%
	\right) \right) ,  \notag \\
	s_{4}(h,k) &=&\sum\limits_{j\text{mod}k}(-1)^{[\frac{jh}{k}]},  \notag \\
	s_{5}(h,k) &=&\sum\limits_{j\text{mod}k}(-1)^{j+[\frac{hj}{k}]}\left( \left( 
	\frac{j}{k}\right) \right) .  \notag
\end{eqnarray}%
We also note that some authors have called Hardy sums as Hardy-Berndt sums.
For $s_{5}(h,k)$, the below equality also holds true:%
\begin{equation}
	s_{5}(h,k)=\frac{1}{k}\sum_{j=1}^{k-1}j(-1)^{j+\left[ \frac{hj}{k}\right] }
	\label{eq3}
\end{equation}%
when $h$ and $k$ are odd, \cite{Berndt and goldberg}. Further, following
equations will be necessary in the next section, \cite{Pettet and
	sitaramachan}:%
\begin{eqnarray}
	\sum_{j=1}^{k-1}(-1)^{j+\left[ \frac{hj}{k}\right] }\left( \frac{j}{k}%
	\right) &=&s_{5}(h,k)-\frac{1}{2}S(h,k),  \label{eq4} \\
	\sum_{j=1}^{h-1}(-1)^{j+\left[ \frac{kj}{h}\right] }\left( \frac{j}{h}%
	\right) &=&s_{5}(k,h)-\frac{1}{2}S(k,h).  \notag
\end{eqnarray}%
Reciprocity law for the $s_{5}(h,k)$ is given by the following theorem:
\begin{theorem}
	\label{th1}Let $h$ and $k$ be coprime positive integers. If $h$ and $k$ are
	odd, then%
	\begin{equation}
		s_{5}(h,k)+s_{5}(k,h)=\frac{1}{2}-\frac{1}{2hk},  \label{eq5}
	\end{equation}%
	and if $h+k$ is odd then%
	\begin{equation}
		s_{5}(h,k)=s_{5}(k,h)=0.  \label{eqas}
	\end{equation}
\end{theorem}

(\textit{cf}. \cite{Apostol-Vu}, \cite{Berndt}, \cite{Berndt and goldberg}, 
\cite{Goldberg}, \cite{Hardy}, \cite{Sitaramachandrarao} and the references
cited in each of these works).

The proof of the next reciprocity theorem was given by Hardy \cite{Hardy}
and Berndt \cite{Berndt3}:
\begin{theorem}
	Let $h$ and $k$ be coprime positive integers. Then%
	\begin{equation}
		S(h,k)+S(k,h)=1\text{ \ \ if \ \ }h+k\text{ is odd.}  \label{qsqt}
	\end{equation}
\end{theorem}

In the light of equation (\ref{qsqt}), Apostol \cite{Apostol-Vu} gave the
below result:

\begin{theorem}
	If both $h$ and $k$ are odd and $(h,k)=1$, then%
	\begin{equation}
		S(h,k)=S(k,h)=0.  \label{eq6}
	\end{equation}
\end{theorem}

The following two theorems give the relations between the Hardy-Berndt sums
and the Dedekind sums $s(h,k)$:
\begin{theorem}
	\cite{Pettet and sitaramachan} Let $(h,k)=1$. Then 
	\begin{equation}
		S(h,k)=8s(h,2k)+8s(2h,k)-20s(h,k),\text{ if \ }h+k\text{ is odd,}
		\label{ef1}
	\end{equation}%
	\begin{equation}
		s_{1}(h,k)=2s(h,k)-4s(h,2k),\text{ if }h\text{ is even,}  \label{ef2}
	\end{equation}%
	\begin{equation}
		s_{2}(h,k)=-s(h,k)+2s(2h,k),\text{ if }k\text{ is even,}  \label{ef3}
	\end{equation}%
	\begin{equation}
		s_{3}(h,k)=2s(h,k)-4s(2h,k),\text{ if }k\text{ is odd,}  \label{ef4}
	\end{equation}%
	\begin{equation}
		s_{4}(h,k)=-4s(h,k)+8s(h,2k),\text{ if }h\text{ is odd,}  \label{ef5}
	\end{equation}%
	\begin{equation}
		s_{5}(h,k)=-10s(h,k)+4s(2h,k)+4s(h,2k),\text{ if }h+k\text{ is even and}
		\label{ef6}
	\end{equation}%
	Each one of $S(h,k)$ ($h+k$ even), $s_{1}(h,k)$ ($h$ odd), $s_{2}(h,k)$ ($k$
	odd), $s_{3}(h,k)$ ($k$ even), $s_{4}(h,k)$ ($h$ even), $s_{5}(h,k)$ ($h+k$
	odd) is zero.
\end{theorem}
\begin{theorem}
	\cite{Sitaramachandrarao} For $h+k$ is odd and $(h,k)=1$ with $k>0,$ then%
	\begin{equation}
		S(h,k)=4s(h,k)-8s(h+k,2k).  \label{ewss}
	\end{equation}
\end{theorem}

Next theorem gives infinite series representation of the Hardy-Berndt sums:

\begin{theorem}
\cite{Berndt and goldberg} Let $h$ and $k$ denote relatively prime integers with $k>0$. If $h+k$ is odd, then
\begin{equation}
S(h,k)=\frac{4}{\pi }\sum_{n=1}^{\infty }\frac{1}{2n-1}\tan \left( \frac{\pi
	h(2n-1)}{2k}\right) ;  \label{ew1}
\end{equation}
if $h$ is even and $k$ is odd, then
\begin{equation}
s_{1}(h,k)=-\frac{2}{\pi}\sum_{\substack{n=1\\2n-1\not\equiv 0 (\textrm{mod}\ k)}}^{\infty}\frac{1}{2n-1}\cot\left(\ \frac{\pi h\left(2n-1\right)}{2k}\right); \label{ew2}
\end{equation}
if $h$ is odd and $k$ is even, then
\begin{equation}
s_{2}(h,k)=-\frac{1}{2\pi}\sum_{\substack{n=1\\2n\not\equiv 0(\textrm{mod}\ k)}}^{\infty}\frac{1}{n}\tan\left(\frac{\pi hn}{k}\right); \label{ew3}
\end{equation}
if $k$ is odd, then
\begin{equation}
s_{3}(h,k)=\frac{1}{\pi}\sum_{n=1}^{\infty}\frac{1}{n}\tan\left(\frac{\pi hn}{k}\right); \label{ew4}
\end{equation}
if $h$ is odd, then
\begin{equation}
s_{4}(h,k)=\frac{4}{\pi}\sum_{n=1}^{\infty}\frac{1}{2n-1}\cot\left(\frac{\pi h(2n-1)}{2k}\right); \label{ew5}
\end{equation}
and if $h$ and $k$ are odd, then
\begin{equation}
s_{5}(h,k)=\frac{2}{\pi}\sum_{\substack{n=1\\2n-1\not\equiv 0(\textrm{mod}\ k)}}^\infty\frac{1}{2n-1}\tan\left( \frac{\pi h(2n-1)}{2k}\right). \label{ew6}
\end{equation}
\end{theorem}
Now we will give the finite series representation of the Hardy-Berndt sums:
\begin{theorem}
\cite{Berndt and goldberg} Let $h$ and $k$ be coprime integers with $k>0$.
If $h+k$ is odd, then
\begin{equation}
S(h,k)=\frac{1}{k}\sum_{j=1}^{k} \tan \left( \frac{\pi h(2j-1)}{2k}\right)\cot\left(\frac{\pi(2j-1)}{2k}\right), \label{ress}
\end{equation}
if $h$ is even and $k$ is odd, then
\begin{equation*}
s_1(h,k)=-\frac{1}{2k}\sum_{\substack{j=1\\j\neq\frac{k+1}{2}}}^{k}\cot\left(\frac{\pi h(2j-1)}{2k}\right)\cot\left(\frac{\pi (2j-1)}{2k}\right),
\end{equation*}
if $h$ is odd and $k$ is even, then
\begin{equation*}
s_2(h,k)=-\frac{1}{4k}\sum_{\substack{j=1\\j \neq \frac{k}{2}}}^{k-1}\tan\left(\frac{\pi hj}{k}\right)\cot\left(\frac{\pi j}{k}\right),
\end{equation*}
if $k$ is odd, then
\begin{equation*}
s_{3}(h,k)=\frac{1}{2k}\sum_{j=1}^{k-1}\tan \left( \frac{\pi hj}{k}\right)
\cot \left( \frac{\pi j}{k}\right) ,
\end{equation*}
if $h$ is odd, then 
\begin{equation*}
s_{4}(h,k)=\frac{1}{k}\sum_{j=1}^{k}\cot \left( \frac{\pi h(2j-1)}{2k}%
\right) \cot \left( \frac{\pi (2j-1)}{2k}\right) ,
\end{equation*}%
and if $h$ and $k$ are odd, then%
\begin{equation}
s_{5}(h,k)=\frac{1}{2k}\sum_{\substack{ j=1  \\ j\neq \frac{k+1}{2}}}^{k}\tan
\left( \frac{\pi h(2j-1)}{2k}\right) \cot \left( \frac{\pi (2j-1)}{2k}
\right) .  \label{ress2}
\end{equation}
\end{theorem}
In \cite{TJM}, Simsek gave the following new sums: Let $h$ is an integer and$%
\ k$ is a positive integer with $(h,k)=1$. Then
\begin{equation*}
Y(h,k)=4k\sum\limits_{j\text{mod}k}(-1)^{j+[\frac{hj}{k}]}\left( \left( 
\frac{j}{k}\right) \right) .
\end{equation*}
We observe that $Y(h,k)$ sums are also related to the Hardy sums $s_{5}(h,k)$. That is
\begin{equation}
Y(h,k)=4ks_{5}(h,k).  \label{stay}
\end{equation}
Reciprocity law for this sum was given by Simsek in \cite[p. 5, Theorem 4]{TJM} as below:
\begin{equation}
hY(h,k)+kY(k,h)=2hk-2.  \label{eqq}
\end{equation}
Simsek gave two different proofs of this reciprocity law. Another proof was also given in \cite{csc}.

In this paper we study the Hardy sums, the Simsek sums $Y(h,k)$ and the
Dedekind sums $s(h,k)$ which are related to the symmetric pairs, \cite{jeffrey}, and the Fibonacci numbers. Before starting our results, we need some properties of the Fibonacci numbers which are given as follows:
The Fibonacci numbers are defined by means of the following generating function \cite{fibonacci}:
\begin{equation}
F(x)=\frac{x}{1-x-x^{2}}=\sum_{n=0}^{\infty }F_{n}x^{n}.  \label{F1}
\end{equation}
One can easily derive the following recurrence relation from (\ref{F1}):
\begin{equation*}
F_{n+1}=F_{n}+F_{n-1}.
\end{equation*}
From (\ref{F1}), we also easily compute the first few Fibonacci numbers as
follows: $0,1,1,2,3,5,8,13,21,\cdots $
In \cite{jeffrey}, Meyer studied a special case of the Dedekind sums. In that paper, Meyer investigated the pairs of integers $\left\{ h,k\right\} $ for which $s(h,k)=s(k,h)$. Meyer defined that $\left\{ h,k\right\} $ is a symmetric pair if this property holds and he showed that $\left\{h,k\right\} $ is a symmetric pair if and only if $h=F_{2n+1}$ and $k=F_{2n+3} $ for $n\in \mathbb{N}$ where $F_{m}$ is the $m-$th Fibonacci number. In \cite{jeffrey}, Meyer proved the following theorem:
\begin{theorem}
\label{th3}If $(h,k)=1$ and $\left\{ h,k\right\} $ is a symmetric pair, then 
$s(h,k)=0$.
\end{theorem}

In \cite{csc}, Cetin et al. defined the sum $C_{1}(h,k)$ as follows:
\begin{equation}
C_{1}(h,k)=\sum_{j=1}^{k-1}\left( \left( \frac{hj}{k}\right) \right) (-1)^{j+\left[ \frac{hj}{k}\right] } \label{basanti} 
\end{equation}
where $h,k$ are positive integers with $(h,k)=1.$

For the odd values of $k$, the below theorem is given in \cite{cetin}:
\begin{theorem}
If $(h,k)=1,$ $h$ and $k$ are odd integers with $k>0$, then we have%
\begin{equation}
C_{1}(h,k)=\frac{1}{2}-\frac{1}{2k}.  \label{aww}
\end{equation}
\end{theorem}

In \cite{csc}, Cetin et al. also defined the sum $Y_{n-1}(a_{1}, a_{1},\cdots;a_{n})$ as follows:
\begin{equation*}
Y_{n-1}(a_{1}, a_{1},\cdots;a_{n})=\sum_{j=1}^{a_{n}-1}(2j-1)(-1)^{j+\left[\frac{a_{1}j}{a_{n}}\right]+\cdots+\left[\frac{a_{n-1}j}{a_{n}}\right]}\left[\frac{a_{1}j}{a_{n}}\right]\cdots\left[\frac{a_{n-1}j}{a_{n}}\right].
\end{equation*}
where $ a_{1},a_{2} ,\cdots,a_{n}$ are pairwise positive integers.

Two-term polynomial relation will has an important role in the next section.
So we remind it in the next theorem.
\begin{theorem}
If $a$, $b$, and $c$ are pairwise coprime positive integers, then
\begin{equation}
(u-1)\sum_{x=1}^{a-1}u^{x-1}v^{\left[ \frac{bx}{a}\right] }+(v-1)%
\sum_{y=1}^{b-1}v^{y-1}u^{\left[ \frac{ay}{b}\right] }=u^{a-1}v^{b-1}-1.
\label{aaa}
\end{equation}
\end{theorem}
Equation (\ref{aaa}) is originally due to Berndt and Dieter \cite{Berndt2}.

Next corollary, which was given in \cite[Corollary 7]{cetin}, will be useful
in the next section.

\begin{corollary}
Let $h$ and $\ k$ be positive integers and $\left\{ h,k\right\} $ is a
symmetric pair. If $(h,k)=1$, $h=F_{6n-1}$ and $k=F_{6n+1}$ with $n$ is a
natural number, where $F_{m}$ is the $m-$th Fibonacci number, then%
\begin{equation}
s_{5}(h,k)+s_{5}(k,h)=\frac{1}{2}\left( \frac{h}{k}+\frac{k}{h}-2\right) ,
\label{kuch}
\end{equation}
and
\begin{equation*}
hY(h,k)+kY(k,h)=2h^{2}+2k^{2}-4hk.
\end{equation*}
\end{corollary}

\section{The Sum $B_{1}(h,k)$ and Its Properties}

In \cite{csc} we defined a new sum as follows:
\begin{equation*}
B_{1}(h,k)=\sum_{j=1}^{k-1}(-1)^{j+\left[ \frac{hj}{k}\right] }\left[ \frac{%
hj}{k}\right] ,
\end{equation*}
which $(h,k)=1$ and $k>0$. The sum $B_{1}(h,k)$ has the following arithmetic
property:
\begin{equation}
B_{1}(-h,-k)=B_{1}(h,k).  \label{rangde}
\end{equation}
To show that the last equality holds true, we use the definition of the $%
\left[ .\right] $ function, and the fact that $\left( \left( -x\right)
\right) =-\left( \left( x\right) \right) $. If we also consider the equation 
\begin{equation}
(-1)^{\left[ x\right] }=2\left( \left( x\right) \right) -4\left( \left( 
\frac{x}{2}\right) \right)   \label{peela}
\end{equation}
when $x$ is not an integer, then we get the equation (\ref{rangde}). The
equation (\ref{peela}) is originally from \cite{Sitaramachandrarao}.
\\Now we will give a relation between the sums $B_{1}(h,k)$ and $S(h,k)$.
\begin{theorem}
\label{the1}If $h+k$ is odd, $k>0$, and $(h,k)=1$, then%
\begin{equation}
B_{1}(h,k)=\frac{1}{2}(1-h)S(h,k).  \label{eqa7}
\end{equation}
\end{theorem}
\begin{proof}
We consider the two-term relation which is given in (\ref{aaa}). If we take
the partial derivative of (\ref{aaa}) with respect to $u$, and substitute $u=v=-1,$ then we have
\begin{equation*}
\sum_{x=1}^{h-1}(-1)^{x+\left[ \frac{kx}{h}\right] }-2%
\sum_{x=1}^{h-1}x(-1)^{x+\left[ \frac{kx}{h}\right] }-2\sum_{y=1}^{k-1}\left[
\frac{hy}{k}\right] (-1)^{y+\left[ \frac{hy}{k}\right] }=(h-1)(-1)^{h+k-1}.
\end{equation*}
After some elementary calculations and by using (\ref{eq4}), then we get
\begin{equation*}
-S(h,k)-2h\left( s_{5}(k,h)-\frac{1}{2}S(k,h)\right) -2B_{1}(h,k)=h-1.
\end{equation*}
We know frrom (\ref{eqas}) that $s_{5}(h,k)=s_{5}(k,h)=0.$ If we use this
fact, then we have 
\begin{equation}
-2B_{1}(h,k)=(h-1)(1-S(k,h)).  \label{eqwest}
\end{equation}
From (\ref{qsqt}) we can write
\begin{equation}
S(h,k)=1-S(k,h).  \label{wsd}
\end{equation}
If we put (\ref{wsd}) into (\ref{eqwest}), then we have desired result.
\end{proof}
In the next theorem, we will give the relation between the sums $B_{1}(h,k)$
and the Hardy-Berndt sums $s_{5}(h,k)$:

\begin{theorem}
\label{the7}If $h$ and $k$ are relatively prime odd numbers with $k>0$, then%
\begin{equation}
B_{1}(h,k)=hs_{5}(h,k)+\frac{1}{2k}-\frac{1}{2}.  \label{rff}
\end{equation}
\end{theorem}

\begin{proof}
From the definition of the sum $B_{1}(h,k)$ with basic calculations, we get%
\begin{equation*}
B_{1}(h,k)=hs_{5}(h,k)-C_{1}(h,k)+\frac{1}{2}S(h,k).
\end{equation*}%
From \cite{cetin}, we know that equation (\ref{aww}) holds true. If we also
use (\ref{eq6}), then we get the desired result.
\end{proof}
In the next theorem we will give the relation between the sums $B_{1}(h,k)$
and the sums $Y(h,k)$:

\begin{theorem}
\label{the12}If $h$ and $k$ are relatively prime odd numbers with $k>0$, then%
\begin{equation*}
B_{1}(h,k)=\frac{h}{4k}Y(h,k)+\frac{1}{2k}-\frac{1}{2}.
\end{equation*}
\end{theorem}

\begin{proof}
It can be directly obtain from Theorem \ref{the7} and equation (\ref{stay}).
\end{proof}

Now, we will give a relation for the sums $B_{1}(h,k)$ as follows:

\begin{theorem}
\label{the2}If $h+k$ is an odd positive integer and $(h,k)=1$, then%
\begin{equation}
(k-1)B_{1}(h,k)+(h-1)B_{1}(k,h)=-\frac{1}{2}(k-1)(h-1). \label{eqa111}
\end{equation}
\end{theorem}
\begin{proof}
From Theorem \ref{the1}, we showed that equation (\ref{eqa7}) holds.
Similarly, when $h+k$ is an odd positive integer with $h>0,$ we can also
write%
\begin{equation}
B_{1}(k,h)=\frac{1}{2}(1-k)S(h,k).  \label{eqa8}
\end{equation}%
So first, if we multiply (\ref{eqa7}) by $k$ and (\ref{eqa8}) by $h$
respectively, then if we sum the two equations side by side and use (\ref%
{qsqt}), we get the following identity: 
\begin{equation}
kB_{1}(h,k)+hB_{1}(k,h)=\frac{1}{2}\left[ kS(h,k)+hS(k,h)-hk\right] .
\label{eqa10}
\end{equation}%
Now we will consider the sum $B_{1}(h,k)+B_{1}(k,h)$. From (\ref{eqa7}) and (%
\ref{eqa8}), we can see that%
\begin{equation*}
B_{1}(h,k)+B_{1}(k,h)=\frac{1}{2}\left[ 1-\left( hS(h,k)+kS(k,h)\right) %
\right] .
\end{equation*}%
So from this last equation, we can write%
\begin{equation}
hS(h,k)+kS(k,h)=1-2B_{1}(h,k)-2B_{1}(k,h).  \label{eqa11}
\end{equation}%
Now we will use the equation (\ref{qsqt}). First, we multiply (\ref{qsqt})
by $h$, and we multiply (\ref{qsqt}) by $k$. Then if add these two equations
side by side and if we use the (\ref{eqa10}), we get the desired result.
\end{proof}
In the below theorem, we will give the reciprocity theorem for the sums $%
B_{1}(h,k)$:

\begin{theorem}
\label{the8}If $h$ and $k$ are odd positive integers with $(h,k)=1$, then%
\begin{equation*}
kB_{1}(h,k)+hB_{1}(k,h)=\frac{1}{2}(h-1)(k-1).
\end{equation*}
\end{theorem}

\begin{proof}
From Theorem \ref{the7}, we know that equation (\ref{rff}) holds. Similarly,
we can also write%
\begin{equation}
B_{1}(k,h)=ks_{5}(k,h)+\frac{1}{2h}-\frac{1}{2}.  \label{pfff}
\end{equation}%
If we multiply equation (\ref{rff}) by $k$, and equation (\ref{pfff}) by $h$
respectively, and add these equations side by side, we get%
\begin{equation}
kB_{1}(h,k)+hB_{1}(k,h)=hk\left( s_{5}(h,k)+s_{5}(k,h)\right) +1-\frac{k}{2}-%
\frac{h}{2}.  \label{sanam}
\end{equation}%
In this last equation if we use (\ref{eq5}), then we get the desired result.
\end{proof}

 \begin{theorem}
 If $h+k$ is an odd positive integer and $(h,k)=1$, then%
\begin{equation*}
Y_{1}(h,k)+Y_{1}(k,h)=2(k-1)B_{1}(h,k)+2(h-1)B_{1}(k,h).
\end{equation*}
 \end{theorem}
\begin{proof}
In \cite{csc}, if we take $ n=2 $ in Theorem 4, and use it with equation (\ref{eqa111}) we have desired result.
\end{proof}

\begin{theorem}
 If $h+k$ is an odd positive integer and $(h,k)=1$, then%
\begin{equation}
Y_{1}(h,k)+Y_{1}(k,h)=2kB_{1}(h,k)+2hB_{1}(k,h).  \label{eqa112}
\end{equation}
 \end{theorem}
\begin{proof}
In \cite{csc}, if we take $ n=2 $ in Theorem 4, and use it with equation (\ref{eqa112}) we have desired result.
\end{proof}

In the following three theorems, we will give the relations between the sums 
$B_{1}(h,k)$ and the Dedekind sums $s(h,k)$:

\begin{theorem}
\label{the3}Let $h+k$ is odd, $(h,k)=1$ with $k>0$. Then%
\begin{equation*}
B_{1}(h,k)=(1-h)\left( 4s(h,2k)+4s(2h,k)-10s(h,k)\right) .
\end{equation*}
\end{theorem}

\begin{proof}
It can be directly obtain from Theorem \ref{the1} and equation (\ref{ef1}).
\end{proof}

\begin{theorem}
\label{the4}Let $h+k$ is odd, $(h,k)=1$ with $k>0$. Then%
\begin{equation*}
B_{1}(h,k)=2(1-h)\left( s(h,k)-2s(h+k,2k)\right) .
\end{equation*}
\end{theorem}

\begin{proof}
It can be directly obtain from Theorem \ref{the1} and equation (\ref{ewss}).
\end{proof}
\begin{theorem}
\label{the9}If $h$ and $k$ are relatively prime odd numbers with $k>0,$ then%
\begin{equation*}
B_{1}(h,k)=-10hs(h,k)+4hs(2h,k)+4hs(h,2k)+\frac{1}{2k}-\frac{1}{2}.
\end{equation*}
\end{theorem}

\begin{proof}
It can be directly obtain from equation (\ref{rff}) and equation (\ref{ef6}).
\end{proof}

Now we will give two different infinite series representations of the sums $%
B_{1}(h,k)$:

\begin{theorem}
\label{the5}Let $h$ and $k$ denote relatively prime integers with $k>0$. If $%
h+k$ is odd, then%
\begin{equation*}
B_{1}(h,k)=\frac{2(1-h)}{\pi }\sum_{n=1}^{\infty }\frac{1}{2n-1}\tan \left( 
\frac{\pi h(2n-1)}{2k}\right) .
\end{equation*}
\end{theorem}

\begin{proof}
It can be directly obtain from Theorem \ref{the1} and equation (\ref{ew1}).
\end{proof}

\begin{theorem}
\label{the10}Let $h$ and $k$ denote relatively prime integers with $k>0$. If 
$h$ and $k$ are odd, then%
\begin{equation*}
B_{1}(h,k)=\frac{2h}{\pi }\sum_{\substack{ n=1  \\ 2n-1\not \equiv \text{mod}k}}%
^{\infty }\frac{1}{2n-1}\tan \left( \frac{\pi h(2n-1)}{2k}\right) +\frac{1}{2k}-\frac{1}{2}.
\end{equation*}
\end{theorem}

\begin{proof}
It can be directly obtain from Theorem \ref{the7} and equation (\ref{ew6}).
\end{proof}
Similarly we give two different finite series representations of the sums $B_{1}(h,k)$ in the below:
\begin{theorem}
\label{the6}Let $h$ and $k$ denote relatively prime integers with $k>0$. If $h+k$ is odd, then
\begin{equation*}
B_{1}(h,k)=\frac{1-h}{2k}\sum_{j=1}^{k}\tan \left( \frac{\pi h(2j-1)}{2k}%
\right) \cot \left( \frac{\pi (2j-1)}{2k}\right) .
\end{equation*}
\end{theorem}

\begin{proof}
It can be directly obtain from Theorem \ref{the1} and equation (\ref{ress}).
\end{proof}

\begin{theorem}
\label{the11}Let $h$ and $k$ be coprime integers with $k>0$. If $h$ and $k$
are odd, then
\begin{equation*}
B_{1}(h,k)=\frac{h}{2k}\sum_{\substack{ j=1  \\ j\neq \frac{k+1}{2}}}^{k}\tan
\left( \frac{\pi h(2j-1)}{2k}\right) \cot \left( \frac{\pi (2j-1)}{2k}%
\right) +\frac{1}{2k}-\frac{1}{2}.
\end{equation*}
\end{theorem}

\begin{proof}
It can be directly obtain from Theorem \ref{the7} and equation (\ref{ress2}).
\end{proof}

Now, we will give the relation between the sums $B_{1}(h,k)$ and the
Fibonacci numbers.
\begin{theorem}
\label{the14} Let $h$ and $\ k$ be positive integers and $\left\{
h,k\right\} $ is a symmetric pair. If $(h,k)=1$, $h=F_{6n-1}$ and $%
k=F_{6n+1} $ with $n$ is a natural number, where $F_{m}$ is the $m-$th
Fibonacci number, then%
\begin{equation*}
kB_{1}(h,k)+hB_{1}(k,h)=\frac{h^{2}-h-k+k^{2}}{2}-hk+1.
\end{equation*}
\end{theorem}

\begin{proof}
It can be obtained similarly with Theorem \ref{the8}'s proof. From Theorem 
\ref{the8}'s proof, we know that equation (\ref{sanam}) holds. If we also
use the equation (\ref{kuch}) into equation (\ref{sanam}), then we get
desired result.
\end{proof}

\end{document}